\documentclass[a4paper,12pt]{article}
\usepackage[english]{babel}
\usepackage{graphicx}
\usepackage{amssymb}
\usepackage{amsthm}
\usepackage{setspace}

\usepackage{amsmath,amsthm,amscd,amsfonts,mathrsfs,amssymb}
\usepackage{graphicx,enumerate}

\usepackage{mathpazo}

\usepackage[all]{xy}
\usepackage{fullpage} 
\usepackage{color}
\usepackage{makeidx}
\usepackage{alltt}
\usepackage{setspace}
\usepackage{hyperref}

\newtheorem{theorem}{Theorem}[section]
\newtheorem{lemma}[theorem]{Lemma}

\theoremstyle{definition}

\theoremstyle{definition}

\newcommand{\dd}{\;{\rm{d}}}
\newcommand{\ssl}[2]{\textrm{SL}(#1,\mathbb{#2})}

\newcommand{\LL}{$L$-}

\newcommand{\bra}[1]{\left(#1\right)}

\newcommand{\sumsom}{\underset{\psi\mod q}{\sum {}^*}}
\newcommand{\sumeven}{\underset{\underset{\scriptstyle{ \psi(-1)=1}}{\psi\mod q}}{\sum {}^*}}
\newcommand{\sumodd}{\underset{\underset{\scriptstyle{ \psi(-1)=-1}}{\psi\mod q}}{\sum {}^*}}

\newcommand{\aaa}{\overline{a}}

\newcommand{\psibar}{\bar{\psi}}
\newcommand{\gauss}{\tau(\psi)}
\newcommand{\gaussbar}{\tau(\bar\psi)}

\newcommand{\summ}{\sum_{m=1}^\infty}
\newcommand{\lfa}{L(s,\pi\times \psi)}
\newcommand{\lfaa}{L(s,\pi)}
\newcommand{\lfb}{L(s,\tilde\pi \times\psibar)}
\newcommand{\lfbb}{L(s,\tilde\pi)}

\newcommand{\tth}{^{\text{th}}}
\newcommand{\aleft}[1]{A(1,\dots,1,\overset{ \text{position } #1 }{\overbrace{q,1,\dots,1}})}
\newcommand{\reff}[1]{(\ref{#1})}

\begin{document}

\title{\scshape Voronoi Summation Formulae on GL($n$)}

\author{Fan Zhou}
\maketitle

\begin{abstract} 
We discover new Voronoi formulae for automorphic forms on GL($n$) for $n\geq 4$. There are $[n/2]$ different Voronoi formulae on GL($n$), which are Poisson summation formulae weighted by Fourier coefficients of the automorphic form with twists by some hyper-Kloosterman sums.
\end{abstract}

\section{Introduction}
A Voronoi summation formula of an automorphic form is a 
Poisson summation formula weighted by Fourier coefficients of the automorphic form with twists by some additive arithmetic functions. 
The Voronoi formula of an automorphic form on GL(2) is well known and has a lot of application.

A Voronoi formula of an automorphic form on GL(3) was firstly proved  by Miller and Schmid in \cite{millerschmid1}. Later, a Voronoi formula was established for GL($n$) for $n\geq 4$ in \cite{millerschmid2}, \cite{goldfeldli1}, and \cite{goldfeldli2}. 
An adelic version was established in \cite{ichinotemplier}.
The Voronoi formula has important application, such as a nontrivial bound on the exponential sum in  \cite{miller} and subconvexity bounds in  \cite{lisubconvexity}.

More recently, Li and Miller established a different type of Voronoi formula on GL(4)\footnote{Their result was announced at
\textit{Representation Theory, Automorphic Forms, and Complex Geometry, a conference in honor of the 70th birthday of Wilfried Schmid
} in May 2013. The slides of Li's talk are available in the link: \url{http://www.math.harvard.edu/conferences/schmid_2013/docs/li.pdf}}. Their formula is an equality between a weighted sum of Fourier coefficients of an automorphic form on GL(4) 
and a dual sum, both twisted by classical Kloosterman sums. This formula is different from those in \cite{millerschmid2}, \cite{goldfeldli1}, \cite{goldfeldli2}, and \cite{ichinotemplier}
and inspires us to develop this paper.

In this paper, we discover more Voronoi formulae on GL($n$) when $n\geq 4$. One of these new formulae overlaps the formula of Li and Miller. Our formula on GL($n$) is an equality between a sum of Fourier coefficients of an automorphic form twisted by hyper-Kloosterman sums of dimension $(k-1)$ and a dual sum twisted by hyper-Kloosterman sums of dimension $(n-k-1)$. The Voronoi formula of \cite{millerschmid1}, \cite{millerschmid2}, \cite{goldfeldli1}, and \cite{goldfeldli2} are the case of $k=1$. Li and Miller's Voronoi formula on GL(4) is the case of $n=4$ and $k=2$.
For an automorphic form on GL($n$), there are $[n/2]$ different Voronoi formulae, which correspond to $k=1,\dots,[n/2]$ respectively.

For an integer $k\geq 1$ define the $(k-1)$-dimensional hyper-Kloosterman sum by
$$Kl_k(m,q)=\sum_{\underset{\scriptstyle (x_1,q)=\dots=(x_{k-1},q)=1}{x_1, \dots, x_{k-1}\mod q}} e\bra{\frac{x_1+\dots+x_{k-1}+m\overline{x_1\dots x_{k-1}}}{q}}$$ for a prime number $q$.
We state our formulae in the following.

\begin{theorem}\label{main}
Let $\pi$ be an even Maass cusp form for $\ssl{n}{Z}$ 
of the spectral parameter $(\lambda_1,\dots,\lambda_n)\in\mathbb{C}^n$ and $n\geq 2$. Let $A(m_1,\dots,m_{n-1})$ be the Fourier coefficient of $\pi$. Let 
$\omega\in C_c^\infty(0,\infty)$ be a test function with $\tilde{\omega}$ its Mellin transform
and let $q$ be a prime number, $a$ and $\aaa$ two integers with $a\aaa\equiv 1\mod q$. 
For $1\leq k \leq n-1,$ we have 
\begin{align*}
&\summ A(1,\dots,1,m)Kl_k(am,q)\omega(m)\\
&\quad+\sum_{2\leq l\leq k}\summ(-1)^{l+k}q^{l-1} A(1,\dots,1,\overset{ \text{position }l}{\overbrace{q,1,\dots,1,m}})\omega\bra{{m}{q^l}}\\
\quad&=\frac{q^k}{2} \summ \frac{A(m,1,\dots,1)}{m}\\
\quad&\quad
\times\bra{Kl_{n-k}(\bar am,q)\bra{\Omega_++\Omega_-}\bra{\frac{m}{q^n}}
+Kl_{n-k}(-\bar am,q)\bra{\Omega_+-\Omega_-}\bra{\frac{m}{q^n}}}\\
\quad&\quad+\sum_{2\leq l\leq n-k}\summ (-1)^{n-k+l}q^{k-1} \frac{A(\overset{\text{position } l}{\overbrace{m,1,\dots,1,q}},1,\dots,1)}{m}\Omega_+\bra{\frac{m}{q^{n-l}}},
\end{align*}
where $\Omega_{\pm}$ is defined as $$\Omega_+(x)=\frac{1}{2\pi i}\int_{\Re {s}=-\sigma}\tilde{\omega}(s)\pi^{-n(1/2-s)}\prod_{j=1}^n \Gamma\bra{\frac{1-s-\overline{\lambda_j}}{2}}\Gamma\bra{\frac{s-{\lambda_j}}{2}}^{-1}x^s\dd s$$
and$$
\Omega_-(x)=\frac{1}{2\pi i}\int_{\Re {s}=-\sigma}\tilde{\omega}(s)i^{-n}\pi^{-n(1/2-s)}\prod_{j=1}^n \Gamma\bra{\frac{2-s-\overline{\lambda_j}}{2}}\Gamma\bra{\frac{s+1-{\lambda_j}}{2}}^{-1}x^s\dd s.$$
\end{theorem}
\begin{proof}
This is literally the sum of Theorem \ref{maineven} and Theorem \ref{mainodd}.
\end{proof}

Our proof is based on the functional equations of the twisted \LL functions of $\pi$ by (multiplicative) Dirichlet characters. 
The relationship between the Voronoi formulae and the functional equations of $L$-functions is known to experts.
The same method was used in Section 4 of \cite{goldfeldli1} and in          proving Lemma 1.3 in \cite{buttcane} for a special case on GL(3).

Our proof does not depend on automorphicity but depends only on the family of the functional equations of GL(1)-twisted $L$-functions. 
Functional equations of GL(1) twists are consequences of automorphicity but they do not imply automorphicity on GL($n$) when $n\geq 3$.
In view of the converse theorem on GL($n$) of Cogdell and Piatetski-Shapiro \cite{converse}, a proof based on GL(1)-twisted functional equations is stronger than a proof based on automorphicity.

Our theorem and proof can be easily applied to many conjectural functorial lifts. 
For example, although the Rankin-Selberg convolutions on GL($m$)$\times$GL($n$) are not generally known to be automorphic on GL($m\times n$), except the few cases of GL(2)$\times$GL(2) and GL(2)$\times$GL(3), the functional equations of GL($m$)$\times$GL($n$)$\times$GL(1) are well known. 
Thus, we have the Voronoi formulae for the Rankin-Selberg convolutions on GL($m$)$\times$GL($n$). Considering the Langlands-Shahidi method and the general Rankin-Selberg method, there are numerous functorial cases to have Voronoi formulae without being known to be automorphic.

The obvious weakness of our result is that the denominator $q$ in the hyper-Kloosterman sums has to be a prime number. More general formulae with arbitrary integer $q$ awaits exploration.




\subsection*{Acknowledgment}
The author wants to thank Wenzhi Luo, Roman Holowinsky, and Eyal Kaplan for helpful discussion. The author wants to thank Dorian Goldfeld for introducing this topic to him.

\section{Background}
Let $\pi$ be an even Maass cusp form for $\ssl{n}{Z}$ with the spectral parameter $(\lambda_1,\dots,\lambda_n)\in \mathbb{C}^n$. Let $A(*,\dots,*)$ be the Fourier-Whittaker coefficient of $\pi$. We have $$A(\pm m_1,\dots,\pm m_{n-1})=A(m_1,\dots,m_{n-1}),$$
$$\overline{A(m_1,\dots,m_{n-1})}=A(m_{n-1},\dots,m_1),$$
and $A(1,\dots,1)=1.$ We refer to \cite{goldfeld} for the definitions and the basic results of Maass forms for $\ssl{n}{Z}$.
The dual (contragredient) Maass form of $\pi$ is denoted by $\tilde{\pi}$. Let $B(*,\dots,*)$ be the Fourier-Whittaker coefficient of $\tilde\pi$ and we have 
$$\overline{A(m_1,\dots,m_{n-1})}=B(m_{n-1},\dots,m_1).$$

The standard $L$-function of $\pi$ is $$L(s,\pi)=\summ \frac{A(1,\dots,1,m)}{m^s}.$$
It satisfies the functional equation 
\begin{equation}\label{FEstandard}
\lfaa = G_+(s)L(1-s,\tilde \pi),\end{equation}
where $$G_+(s)= \pi^{-n(1/2-s)}\prod_{j=1}^n \Gamma\bra{\frac{1-s-\overline{\lambda_j}}{2}}\Gamma\bra{\frac{s-{\lambda_j}}{2}}^{-1}.$$

Let $q$ be a prime number. Let $\psi$ be a nontrivial Dirichlet character mod $q$ and $\gauss$ its Gauss sum.
The $\psi$-twisted $L$-function is $$L(s,\pi\times\psi)=\summ \frac{\psi(m)A(1,\dots,1,m)}{m^s}.$$ 

Let $k$ be an integer with $1\leq k \leq n-1$. In the case of $\psi(-1)=1$, the twisted \LL function satisfies the functional equation
$$\lfa =\gauss^{n} q^{-ns}G_+(s)L(1-s,\tilde{\pi}\times \bar \psi),$$ 
which can be rewritten as 
\begin{equation}\label{FEeven}
\gaussbar^k\lfa =\gauss^{n-k} q^{k-ns}G_+(s)L(1-s,\tilde{\pi}\times \bar \psi).
\end{equation}
In the case of $\psi(-1)=-1$, the twisted \LL function satisfies the functional equation
$$\lfa =\gauss^{n} q^{-ns}G_-(s) L(1-s,\tilde{\pi}\times \bar \psi),$$
which can be rewritten as 
\begin{equation}\label{FEodd}
\gaussbar^k\lfa =\gauss^{n-k} q^{k-ns}G_-(s) L(1-s,\tilde{\pi}\times \bar \psi),
\end{equation}
where 
$$G_-(s)=i^{-n}\pi^{-n(1/2-s)}\prod_{j=1}^n \Gamma\bra{\frac{2-s-\overline{\lambda_j}}{2}}\Gamma\bra{\frac{s+1-{\lambda_j}}{2}}^{-1}.$$



\subsection{Hyper-Kloosterman Sums}
Let $q$ be a prime number. For abbreviation, we define $e(x)=\exp(2\pi i x)$.
Let $k$ be a positive integer.
The $(k-1)$-dimensional hyper-Kloosterman sum is defined as
$$Kl_k(m,q)=\sum_{\underset{\scriptstyle (x_1,q)=\dots=(x_{k-1},q)=1}{x_1, \dots, x_{k-1}\mod q}} e\bra{\frac{x_1+\dots+x_{k-1}+m\overline{x_1\dots x_{k-1}}}{q}}$$
and we shall note that if $q|m$,
\begin{equation}\label{kloostermandegenerate0}
Kl_k(m,q)=(-1)^{k-1}.
\end{equation}
In the degenerate case of $k=1$, we have the additive character $$Kl_1(m,q) = e\bra{\frac{m}{q}}.$$ When $k=2$, 
the 1-dimensional hyper-Kloosterman sum is identical to the classical Kloosterman sum.

We define $\sumsom$ as the summation over nontrivial Dirichlet characters $\psi$ of mod $q$.
It is known 
\begin{equation}\label{kloostermaneven}
\sumeven \gauss^k \psibar(m)=
\begin{cases}
\frac{q-1}{2}\bra{Kl_{k}(m,q)+Kl_{k}(-m,q)}-(-1)^k,&\text{ if }(m,q)=1\\
0, &\text{ otherwise}.
\end{cases}
\end{equation}
It is also known
\begin{equation}\label{kloostermanodd}
\sumodd \gauss^k \psibar(m)=\frac{q-1}{2}\bra{Kl_{k}(m,q)-Kl_{k}(-m,q)}.
\end{equation}
We will use the family of (multiplicative) Dirichlet characters to recover the additive characters and the hyper-Kloosterman sums by the two formulae above.

\section{Even Part}
Let $q$ be a prime number, $a$ and $\aaa$ two integers with $a\aaa\equiv 1\mod q$. 
Let $k$ be an integer and $1\leq k \leq n-1$.  

Define 
two Dirichlet series
$$L_q(s,\pi)=\sum_{q|m,m>0}\frac{A(1,\dots,1,m)}{m^s}$$ 
and $$L_q(s,\tilde\pi)=\sum_{q|m,m>0}\frac{A(m,1,\dots,1)}{m^s}.$$ 
Define 
two Dirichlet series
$$\mathsf{Z}(s,q)=\sumeven \gaussbar^k \psi(a) \lfa + (-1)^k\bra{\lfaa-qL_q(s,\pi)}$$
and $$\tilde{\mathsf{Z}}(s,q)=\sumeven \gauss^{n-k} \psi(a)\lfb + (-1)^{n-k}\bra{\lfbb-qL_q(s,\tilde\pi)}.$$

\begin{lemma}\label{zdirichlet}
We have the identities 
$$\mathsf{Z}(s,q)=\frac{1}{2}\summ \bra{Kl_k(am,q)+Kl_k(-am,q)} \frac{A(1,\dots,1,m)}{m^s}$$
and 
$$\tilde{\mathsf{Z}}(s,q)=\frac{1}{2}\summ
\bra{Kl_{n-k}(\bar am,q)+Kl_{n-k}(-\bar am,q)}
 \frac{A(m,1,\dots,1)}{m^{s}}.$$
\end{lemma}

\begin{proof}
This follows from  \reff{kloostermandegenerate0} and \reff{kloostermaneven},
\end{proof}

For $1\leq l \leq n-1$, define a Dirichlet polynomial 
$$H_l(s,q)=\sum_{i=l}^{n-1} (-1)^{i}\frac{A(1,\dots,1,\overset{ \text{position }i}{\overbrace{q,1,\dots,1}})}{q^{is}}+\frac{(-1)^n}{q^{ns}},$$
where $q$ appears on the $i\tth$ position from the right. 
We define another Dirichlet polynomial 
$$\tilde H_l(s,q)=\sum_{i=l}^{n-1} (-1)^{i}\frac{A(\overset{\text{position } i}{\overbrace{1,\dots,1,q}},1,\dots,1)}{q^{is}}+\frac{(-1)^n}{q^{ns}},$$
where $q$ appears on the $i\tth$ position from the left.

\begin{lemma}\label{termsleft}
We have the identities 
\begin{equation}\label{x1}\lfaa H_l(s,q) = (-1)^l\summ \frac{A(1,\dots,1,\overset{\text{position }l}{\overbrace{q,1,\dots,1,m}})}{q^{ls}m^s}\end{equation} 
for $2\leq l\leq n-1$
and 
\begin{equation}\label{x2}
\lfaa H_1(s,q)=-L_q(s,\pi).
\end{equation}
\end{lemma}
\begin{proof}
By the Hecke relation 
\begin{multline*}
\aleft{i}A(1,\dots,1,m)\\=\begin{cases}
A(1,\dots,1,\overset{\text{position }i}{\overbrace{q,1,\dots,1,m}})+
A(1,\dots,1,\overset{\text{position }i+1}{\overbrace{q,1,\dots,1,m/q}}),&\text{ if }q|m\\
A(1,\dots,1,\overset{\text{position }i}{\overbrace{q,1,\dots,1,m}}),&\text{ otherwise,}
\end{cases}
\end{multline*}
we have \reff{x1}
for $2\leq l\leq n-2$
 after multiplying the Dirichlet series $\lfaa H_l(s)$ term by term.
Similarly, by the Hecke relation
$$A(1,\dots,1,q)A(1,\dots,1,m)=
\begin{cases}A(1,\dots,1,mq)+A(1,\dots,1,q,m/q),&\text{ if }q|m\\
A(1,\dots,1,mq),&\text{ otherwise,}
\end{cases}
$$
and 
$$A(q,1,\dots,1)A(1,\dots,1,m)=
\begin{cases}A(q,1,\dots,1,m)+A(1,\dots,1,m/q),&\text{ if }q|m\\
A(q,1,\dots,1,m),&\text{ otherwise,}
\end{cases}
$$
we have \reff{x2}, and \reff{x1} for $l=n-1$, respectively.
\end{proof}

By a similar proof, we have the following lemma for $\tilde{\pi}$.

\begin{lemma}\label{termsright}
We have the identities
$$\lfbb \tilde{H}_l(s,q) = (-1)^l\summ \frac{A(\overset{\text{position } l}{\overbrace{m,1,\dots,1,q}},1,\dots,1)}{q^{ls}m^s}$$
for $2\leq l\leq n-1$
and
$$\lfbb \tilde H_1(s,q)=-  L_q(s,\tilde\pi).$$
\end{lemma}

\begin{lemma}\label{cruciallemma}
We have the following identity between Dirichlet polynomials
\begin{multline*}
\frac{1}{q}+H_1(s,q)+\sum_{2\leq l\leq k}(q^{l-1}-q^{l-2})H_l(s,q)\\
=(-1)^n q^{k-ns}\bra{\frac{1}{q}+\tilde H_1(1-s,q)+\sum_{2\leq l\leq n- k}(q^{l-1}-q^{l-2})\tilde H_l(1-s,q)}.\end{multline*}
\end{lemma}
\begin{proof}
The left side equals
$$\frac{1}{q}+\sum_{i=1}^{n-1} (-1)^{i}q^{\min\{i,k\}-1}\frac{A(1,\dots,1,\overset{ \text{position }i}{\overbrace{q,1,\dots,1}})}{q^{is}}+\frac{(-1)^nq^{k-1}}{q^{ns}}.$$
The right side equals
\begin{align*}
\frac{1}{q}+\sum_{i=1}^{n-1} &(-1)^{i}q^{\min\{i,n-k\}-1}\frac{A(\overset{\text{position } i}{\overbrace{1,\dots,1,q}},1,\dots,1)}{q^{i-is}}+\frac{(-1)^nq^{n-k-1}}{q^{n-ns}}\\
&=\frac{1}{q}+\sum_{i=1}^{n-1} (-1)^{n-i}q^{\min\{n-i,n-k\}-1}\frac{A(1,\dots,1,\overset{ \text{position }i}{\overbrace{q,1,\dots,1}})}{q^{(n-i)-(n-i)s}}+\frac{(-1)^nq^{n-k-1}}{q^{n-ns}}
\\
&=
\frac{(-1)^nq^{-k-1}}{q^{-ns}}+\sum_{i=1}^{n-1} (-1)^{n-i}q^{\min\{0,i-k\}-1}\frac{A(1,\dots,1,\overset{ \text{position }i}{\overbrace{q,1,\dots,1}})}{q^{-(n-i)s}}
+\frac{1}{q}
\end{align*}
After multiplying $(-1)^nq^{k-ns}$, we match term by term with the left side.
\end{proof}

\subsection{Theorem \ref{maineven}}
Let $\omega(x)\in C_c^\infty(0,\infty)$ be a test function. Let
$$\tilde{\omega}(s)=\int_0^\infty\omega(x)x^{s-1} \dd x$$ be the Mellin transform of $\omega$. 
Define 
\begin{equation}\label{omegaplus}
\Omega_+(x)=\frac{1}{2\pi i}\int_{\Re {s}=-\sigma}\tilde{\omega}(s) G_+(s)x^s\dd s.
\end{equation}
We have the Mellin inversion theorem
\begin{equation}\label{mellininversion}
\omega(x)=\frac{1}{2\pi i}\int_{\Re {s}=\sigma} x^{-s}\tilde{\omega}(s)\dd s.
\end{equation}

\begin{theorem}\label{maineven}
Let $\pi$ be an even Maass cusp form for $\ssl{n}{Z}$ with $n\geq 2$. Let $A(m_1,\dots,m_{n-1})$ be the Fourier coefficient of $\pi$.  Let 
$\omega\in C_c^\infty(0,\infty)$ be a test function
and let $q$ be a prime number, $a$ and $\aaa$ two integers with $a\aaa\equiv 1\mod q$. 
For $1\leq k \leq n-1,$ we have 

\begin{align*}
&\frac{1}{2}\summ A(1,\dots,1,m)\bra{Kl_k(am,q)+Kl_k(-am,q)}\omega(m)\\
&+\sum_{2\leq l\leq k}\summ(-1)^{l+k}q^{l-1} A(1,\dots,1,\overset{ \text{position }l}{\overbrace{q,1,\dots,1,m}})\omega\bra{{m}{q^l}}\\
&\quad\quad=\frac{q^k}{2}\summ \frac{A(m,1,\dots,1)}{m}\bra{Kl_{n-k}(\bar am,q)+Kl_{n-k}(-\bar am,q)}\Omega_+\bra{\frac{m}{q^n}}\\
&\quad\quad\quad+\sum_{2\leq l\leq n-k}\summ (-1)^{n-k+l}q^{k-1} \frac{A(\overset{\text{position } l}{\overbrace{m,1,\dots,1,q}},1,\dots,1)}{m}\Omega_+\bra{\frac{m}{q^{n-l}}},
\end{align*}
where $\Omega_+$ is defined in \reff{omegaplus}.
\end{theorem}

\begin{proof}
By the functional equations \reff{FEstandard} and \reff{FEeven} and Lemma \ref{cruciallemma}, we have
\begin{align*}
\mathsf{Z}(&s,q)+(-1)^kq\sum_{2\leq l\leq k}(q^{l-1}-q^{l-2})H_l(s,q)\lfaa\\
&=\sumeven \gaussbar^k\psi(a)\lfa \\
&\quad+(-1)^k\bra{\lfaa-qL_q(s,\pi)+q\sum_{2\leq l\leq k}(q^{l-1}-q^{l-2})H_l(s,q)\lfaa}
\\
&=\sumeven \gaussbar^k\psi(a)\lfa \\&
\quad +(-1)^k{q\bra{\frac{1}{q}+H_1(s,q)+\sum_{2\leq l\leq k}(q^{l-1}-q^{l-2})H_l(s,q)}\lfaa}
\\
&=\sumeven \gauss^{n-k}\psi(a)q^{k-ns}G_+(s)L(1-s,\tilde{\pi}\times\bar{\psi})
\\
&\quad+(-1)^{n-k}q\bra{\frac{1}{q}+\tilde H_1(1-s,q)+\sum_{2\leq l\leq n-k}(q^{l-1}-q^{l-2})\tilde H_l(1-s,q)}q^{k-ns}G_+(s)L(1-s,\tilde{\pi})\\
&= q^{k-ns}G_+(s)\bra{\tilde{\mathsf{Z}}(1-s,q)+(-1)^{n-k}q\sum_{2\leq l\leq n-k}(q^{l-1}-q^{l-2})\tilde H_l(1-s,q)L(1-s,\tilde{\pi})}.\end{align*}

Integrating on both ends with $\tilde{\omega}$ and shifting the integral on the left from $\Re {s}=-\sigma$ to $\Re {s}=\sigma$, we have 
\begin{align*}
&\int_{\Re {s}=\sigma} \bra{\mathsf{Z}(s,q)+(-1)^kq\sum_{2\leq l\leq k}(q^{l-1}-q^{l-2})H_l(s,q)\lfaa} \tilde{\omega}(s) \dd s\\
&=\int_{\Re {s}=-\sigma} q^{k-ns}G_+(s)\times\\
&\quad\;\quad\;\bra{\tilde{\mathsf{Z}}(1-s,q)+(-1)^{n-k}q\sum_{2\leq l\leq n-k}(q^{l-1}-q^{l-2})\tilde H_l(1-s,q)L(1-s,\tilde{\pi})} \tilde{\omega}(s) \dd s.
\end{align*}
Recalling Lemma \ref{zdirichlet}, Lemma \ref{termsleft}, and Lemma  \ref{termsright}, we complete the proof with the Mellin inversion \reff{mellininversion}.
\end{proof}



\section{Odd Part}
Define $$\mathsf{Y}(s,q)=\sumodd \gaussbar^k \psi(a) \lfa $$
and $$\tilde{\mathsf{Y}}(s,q)=\sumodd \gaussbar^{n-k} \psi(a) \lfb.$$
By the function equation \reff{FEodd}, we have 
\begin{equation}\label{odd}
\mathsf{Y}(s,q)
=q^{k-ns}G_-(s)\tilde{\mathsf{Y}}(1-s,q).
\end{equation}

\begin{lemma}\label{zzdirichlet}
We have the identities 
$$\mathsf{Y}(s,q)=\frac{1}{2}\summ \bra{Kl_k(am,q)-Kl_k(-am,q)} \frac{A(1,\dots,1,m)}{m^s}$$
and 
$$\tilde{\mathsf{Y}}(s,q)=\frac{1}{2}\summ
\bra{Kl_{n-k}(\bar am,q)-Kl_{n-k}(-\bar am,q)}
 \frac{A(m,1,\dots,1)}{m^{s}}.$$
\end{lemma}

\begin{proof}
This follows from \reff{FEodd}.
\end{proof}

Let $\omega(x)\in C_c^\infty(0,\infty)$ be a test function. Let
$\tilde{\omega}(s)=\int_0^\infty\omega(x)x^s \dd x/x$ be its 
Mellin transform. 
Define 
\begin{equation}\label{omegaminus}\Omega_{-}(x)=\frac{1}{2\pi i}\int_{\Re {s}=-\sigma}\tilde{\omega}(s) G_{-}(s)x^s\dd s.\end{equation}
Integrating with $\tilde{\omega}$ on both sides of \reff{odd} and shifting the integral on the left to $\Re {s}=\sigma$, we get $$\frac{1}{2\pi i}\int_{\Re {s}=\sigma}\mathsf{Y}(s,q)\tilde{\omega}(s)\dd s=\frac{1}{2\pi i}\int_{\Re {s}=-\sigma} q^{k-ns}G_-(s)\tilde{\mathsf{Y}}(1-s,q)  \tilde{\omega}(s) \dd s,$$
which, along with Lemma \ref{zzdirichlet} and \reff{mellininversion}, gives us the following theorem.

\begin{theorem}\label{mainodd}
Let $\pi$ be an even Maass cusp form for $\ssl{n}{Z}$ with $n\geq 2$. Let $A(m_1,\dots,m_{n-1})$ be the Fourier coefficient of $\pi$.  Let 
$\omega\in C_c^\infty(0,\infty)$ be a test function 
and let $q$ be a prime number, $a$ and $\aaa$ two integers with $a\aaa\equiv 1\mod q$. 
For $1\leq k \leq n-1,$ we have 
\begin{multline*}
\frac{1}{2}\summ A(1,\dots,1,m)\bra{Kl_k(am,q)-Kl_k(-am,q)}\omega(m)\\
=\frac{q^k}{2} \summ \frac{A(m,1,\dots,1)}{m}\bra{Kl_{n-k}(\bar am,q)-Kl_{n-k}(-\bar am,q)}\Omega_-\bra{\frac{m}{q^n}},
\end{multline*}
where $\Omega_{-}$ is defined in \reff{omegaminus}.
\end{theorem}

$$\;$$
\noindent {\scshape{Fan Zhou}} \\
Department of Mathematics\\
The Ohio State University\\
Columbus, OH 43210, USA\\
zhou.1406@math.osu.edu

\begin{thebibliography}{1}




\bibitem[BuKh]{buttcane}Buttcane, Jack, and Rizwanur Khan. "$L^ 4$-norms of Hecke newforms of large level." \textit{arXiv preprint arXiv:1305.1850} (2013). to appear in \textit{Math. Ann.}, DOI: 10.1007/s00208-014-1142-3.

\bibitem[CoPS]{converse}Cogdell, James W., and Ilya I. Piatetski-Shapiro. "Converse theorems, functoriality, and applications to number theory." \textit{Proceedings of the International Congress of Mathematicians}, Vol. II (Beijing, 2002), Higher Ed. Press, Beijing, 2002: 119-128. 






\bibitem[Go]{goldfeld}Goldfeld, Dorian. \textit{Automorphic forms and L-functions for the group GL(n,R)}. Cambridge University Press, 2006.

\bibitem[GoLi1]{goldfeldli1}Goldfeld, Dorian, and Xiaoqing Li. "Voronoi formulas on GL(n)." \textit{International Mathematics Research Notices} 2006 (2006): 86295.

\bibitem[GoLi2]{goldfeldli2}Goldfeld, Dorian, and Xiaoqing Li. "The Voronoi formula for GL(n,Rℝ)." \textit{International Mathematics Research Notices} 2008 (2008): rnm144.

\bibitem[IcTe]{ichinotemplier}Ichino, Atsushi, and Nicolas Templier. "On the Voronoi formula for GL(n)." \textit{American Journal of Mathematics} 135, no. 1 (2013): 65-101.






\bibitem[Li]{lisubconvexity}Li, Xiaoqing. "Bounds for GL(3)xGL(2) L-functions and GL(3) L-functions." \textit{Annals of Mathematics} 173 (2011): 301-336.





\bibitem[Mi]{miller}Miller, Stephen David. "Cancellation in additively twisted sums on GL(n)." \textit{American journal of mathematics} 128, no. 3 (2006): 699-729.


\bibitem[MiSc1]{millerschmid1} Miller, Stephen D., and Wilfried Schmid. "Automorphic distributions, L-functions, and Voronoi summation for GL(3)." \textit{Annals of mathematics} (2006): 423-488.

\bibitem[MiSc2]{millerschmid2} Miller, Stephen D., and Wilfried Schmid.  "A general Voronoi summation formula for GL(n,Z)." In \textit{Geometry
and analysis}. No. 2, volume 18 of Adv. Lect. Math. (ALM), pages 173-224. Int. Press, Somerville, MA, 2011.



\end{thebibliography}
\end{document}